\documentclass[a4paper]{amsart}
\author{Dilip Raghavan}
\address{Department of Mathematics \\
National University of Singapore\\
Singapore 119076}
\email{raghavan@math.nus.edu.sg}
\urladdr{http://www.math.nus.edu.sg/$\sim$raghavan}
\author{Saharon Shelah}
\thanks{Both authors were partially supported by European Research Council grant 338821.
Publication 1XXX on Shelah's list.}
\date{\today}
\subjclass[2010]{03E50, 03E05, 03E35, 54D80}
\keywords{cardinal invariants, almost disjoint family, reaping number, revised GCH}
\title[Two results]{Two results on cardinal invariants at uncountable cardinals}
\usepackage[only, llbracket, rrbracket]{stmaryrd}
\usepackage{amssymb, amsmath, amsthm, mathrsfs, enumitem, amsfonts, latexsym, bbm}
\def\polhk#1{\setbox0=\hbox{#1}{\ooalign{\hidewidth
    \lower1.5ex\hbox{`}\hidewidth\crcr\unhbox0}}}
\newtheorem{Theorem}{Theorem}

\newtheorem{Lemma}[Theorem]{Lemma}
\newtheorem{Cor}[Theorem]{Corollary}

\newtheorem{Question}[Theorem]{Question}

\theoremstyle{definition}
\newtheorem{Def}[Theorem]{Definition}

\theoremstyle{remark}

\newcommand{\restrict}{\upharpoonright}

\newcommand{\bb}{\mathfrak{b}}
\newcommand{\dd}{{\mathfrak{d}}}
\newcommand{\rr}{{\mathfrak{r}}}
\newcommand{\s}{\mathfrak{s}}

\newcommand{\inva}{{\mathfrak{a}}}
\newcommand{\uu}{{\mathfrak{u}}}

\renewcommand{\[}{\left[}
\renewcommand{\]}{\right]}

\newcommand{\PPP}{\mathcal{P}}

\newcommand{\lc}{\left|}
\newcommand{\rc}{\right|}

\newcommand\ZFC{\mathrm{ZFC}}

\newcommand\PCF{\mathrm{PCF}}

\newcommand\GCH{\mathrm{GCH}}

\newcommand\clb{\mathrm{cl}}

\DeclareMathOperator{\otp}{otp}

\DeclareMathOperator{\dom}{dom}

\DeclareMathOperator{\succc}{succ}

\DeclareMathOperator{\cf}{cf}

\DeclareMathOperator{\set}{set}
\newcommand{\Pset}{\mathcal{P}}

\newcommand{\A}{{\mathscr{A}}}

\newcommand{\pr}[2]{\langle #1, #2 \rangle}
\newcommand{\seq}[4]{\langle {#1}_{#2}: #2 #3 #4 \rangle}

\begin{document}
\begin{abstract}
We prove two $\ZFC$ theorems about cardinal invariants above the continuum which are in sharp contrast to well-known facts about these same invariants at the continuum.
It is shown that for an uncountable regular cardinal $\kappa$, $\bb(\kappa) = {\kappa}^{+}$ implies $\inva(\kappa) = {\kappa}^{+}$.
This improves an earlier result of Blass, Hyttinen, and Zhang~\cite{BHZ}.
It is also shown that if $\kappa \geq {\beth}_{\omega}$ is an uncountable regular cardinal, then $\dd(\kappa) \leq \rr(\kappa)$.
This result partially dualizes an earlier theorem of the authors~\cite{cov_z0}.
\end{abstract}
\maketitle
\section{Introduction} \label{sec:intro}
The theory of cardinal invariants at uncountable regular cardinals remains less developed than the theory at $\omega$.
One of the first papers to explore the situation above $\omega$ was by Cummings and Shelah~\cite{sh541}.
In that paper, they considered the direct analogues of the bounding and dominating numbers.
They also considered bounding and domination modulo the club filter, a notion which has no counterpart at $\omega$ but which becomes very natural at uncountable regular cardinals.
Recall the following definitions.
\begin{Def} \label{def:bkappa}
Let $\kappa > \omega$ be a regular cardinal.
 Let $f, g \in {\kappa}^{\kappa}$.
 $f \; {\leq}^{\ast} \; g$ means that $\lc \{\alpha < \kappa: g(\alpha) < f(\alpha)\} \rc < \kappa$ and $f \; {\leq}_{\clb} \; g$ means that $\{\alpha < \kappa: g(\alpha) < f(\alpha)\}$ is non-stationary.
 We say that $F \subset {\kappa}^{\kappa}$ is \emph{$\ast$-unbounded} if $\neg \exists g \in {\kappa}^{\kappa}\forall f \in F\[f \; {\leq}^{\ast} \; g\]$ and we say that $F$ is \emph{$\clb$-unbounded} if $\neg\exists g \in {\kappa}^{\kappa} \forall f \in F\[f \; {\leq}_{\clb} \; g\]$.
 Define
 \begin{align*}
  &\bb(\kappa) = \min\{\lc F \rc: F \subset {\kappa}^{\kappa} \wedge F \ \text{is} \ \ast\text{-unbounded}\},\\
  &{\bb}_{\clb}(\kappa) = \min\{\lc F \rc: F \subset {\kappa}^{\kappa} \wedge F \ \text{is} \ {\clb}\text{-unbounded}\}.
 \end{align*}
 We say that $F \subset {\kappa}^{\kappa}$ is \emph{$\ast$-dominating} if $\forall g \in {\kappa}^{\kappa} \exists f \in F\[g \: {\leq}^{\ast} \: f\]$ and we say that $F$ is \emph{$\clb$-dominating} if $\forall g \in {\kappa}^{\kappa} \exists f \in F\[g \: {\leq}_{\clb} \: f\]$.
 Define
  \begin{align*}
    &\dd(\kappa) = \min\left\{ \lc F \rc: F \subset {\kappa}^{\kappa} \ \text{and} \ F \ \text{is} \ \ast\text{-dominating} \right\}.\\
    &{\dd}_{\clb}(\kappa) = \min\left\{ \lc F \rc: F \subset {\kappa}^{\kappa} \ \text{and} \ F \ \text{is} \ \clb\text{-dominating} \right\}. 
  \end{align*}
\end{Def}
Cummings and Shelah~\cite{sh541} proved that for any regular $\kappa$, ${\kappa}^{+} \leq \cf(\bb(\kappa)) = \bb(\kappa) \leq \cf(\dd(\kappa)) \leq \dd(\kappa) \leq {2}^{\kappa}$, and that these are the only relations between $\bb(\kappa)$ and $\dd(\kappa)$ that are provable in $\ZFC$, thereby generalizing a classical result of Hechler from the case $\kappa = \omega$.
Quite remarkably, they also showed that for every regular $\kappa > \omega$, $\bb(\kappa) = {\bb}_{\clb}(\kappa)$, and that if $\kappa \geq {\beth}_{\omega}$ is regular, then $\dd(\kappa) = {\dd}_{\clb}(\kappa)$.
The question of whether ${\dd}_{\clb}(\kappa) < \dd(\kappa)$ is consistent for any $\kappa$ was left open; as far as we are aware, it remains open.

Other early papers which studied the splitting number at uncountable cardinals revealed interesting differences with the situation at $\omega$.
Recall the following definitions.
\begin{Def} \label{def:rd}
  Let $\kappa > \omega$ be a regular cardinal.
  For $A, B \in \Pset(\kappa)$, $A \: {\subset}^{\ast} \: B$ means $\lc A \setminus B \rc < \kappa$.
  For a family $F \subset {\[\kappa\]}^{\kappa}$ and a set $B \in \Pset(\kappa)$, $B$ is said to \emph{reap} $F$ if for every $A \in F$, $\lc A \cap B \rc = \lc A \cap (\kappa \setminus B) \rc = \kappa$.
  We say that $F \subset {\[\kappa\]}^{\kappa}$ is \emph{unreaped} if there is no $B \in \Pset(\kappa)$ that reaps $F$.
  \begin{align*}
    \rr(\kappa) = \min\left\{ \lc F \rc: F \subset {\[\kappa\]}^{\kappa} \ \text{and} \ F \ \text{is unreaped} \right\}.
  \end{align*}
  A family $F \subset \Pset(\kappa)$ is called a \emph{splitting family} if 
  \begin{align*}
   \forall B \in {\[\kappa\]}^{\kappa} \exists A \in F\[\lc B \cap A \rc = \lc B \cap (\kappa \setminus A) \rc = \kappa\].
  \end{align*}
  \begin{align*}
    \s(\kappa) = \min\left\{ \lc F \rc: F \subset \Pset(\kappa) \ \text{and} \ F \ \text{is a splitting family} \right\}.
  \end{align*}
\end{Def}
For instance, Suzuki~\cite{suzukisplitting} showed that for a regular cardinal $\kappa > \omega$, $\s(\kappa) \geq \kappa$ iff $\kappa$ is strongly inaccessible and $\s(\kappa) \geq {\kappa}^{+}$ iff $\kappa$ is weakly compact.
Zapletal~\cite{jindrasplitting} additionally showed that the statement that there exists some regular uncountable cardinal $\kappa$ for which $\s(\kappa) \geq {\kappa}^{++}$ has large consistency strength, significantly more than a measurable cardinal.
More recently, the authors proved in \cite{cov_z0} that $\s(\kappa) \leq \bb(\kappa)$ for all regular $\kappa > \omega$.
This is in marked contrast to the situation at $\omega$, where it is known that $\s(\omega)$ and $\bb(\omega)$ are independent.
More information about cardinal invariants at $\omega$ can be found in \cite{blasssmall}.

Blass, Hyttinen, and Zhang~\cite{BHZ} is a work about the almost disjointness number at regular uncountable cardinals.
Let us recall the definition of maximal almost disjoint families.
\begin{Def} \label{def:ak}
  Let $\kappa > \omega$ be a regular cardinal.
  $A, B \in {\[\kappa\]}^{\kappa}$ are said to be \emph{almost disjoint} or \emph{a.d.\@} if $\lc A \cap B \rc < \kappa$.
  A family $\A \subset {\[\kappa\]}^{\kappa}$ is said to be \emph{almost disjoint} or \emph{a.d.\@} if the members of $\A$ are pairwise a.d.
  Finally $\A \subset {\[\kappa\]}^{\kappa}$ is called \emph{maximal almost disjoint} or \emph{m.a.d.\@} if $\A$ is an a.d.\@ family, $\lc \A \rc \geq \kappa$, and $\A$ cannot be extended to a larger a.d.\@ family in ${\[\kappa\]}^{\kappa}$.
  \begin{align*}
    \inva(\kappa) = \min\left\{ \lc \A \rc: \A \subset {\[\kappa\]}^{\kappa} \ \text{and} \ \A \ \text{is m.a.d.\@} \right\}.
  \end{align*}
\end{Def}
Blass, Hyttinen, and Zhang~\cite{BHZ} proved that if $\kappa > \omega$ is regular, then $\dd(\kappa) = {\kappa}^{+}$ implies $\inva(\kappa) = {\kappa}^{+}$.
This is potentially different from the situation at $\omega$: it remains an open problem whether $\dd(\omega) = {\aleph}_{1}$ implies $\inva(\omega) = {\aleph}_{1}$, while Shelah~\cite{sh700} showed the consistency of $\dd(\omega) = {\aleph}_{2} < {\aleph}_{3} = \inva(\omega)$ (see also Question \ref{q:b<a}).

There is also a well-developed theory of duality for cardinal invariants at $\omega$.
Thus, for example, $\bb(\omega)$ and $\dd(\omega)$ are dual to each other, while $\s(\omega)$ and $\rr(\omega)$ are duals.
The $\ZFC$ inequality $\s(\omega) \leq \dd(\omega)$ dualizes to the inequality $\bb(\omega) \leq \rr(\omega)$, and indeed even the proof of $\s(\omega) \leq \dd(\omega)$ dualizes to the proof of $\bb(\omega) \leq \rr(\omega)$.
It is possible to make this notion of duality precise using Galois-Tukey connections.
We refer the reader to \cite{blasssmall} for further details about duality of cardinal invariants at $\omega$.
It is unclear at present if there can be a smooth theory of duality for cardinal invariants at uncountable cardinals too.
For example, if we try to na{\"i}vely dualize Suzuki's result mentioned above that $\s(\kappa)$ is small for most $\kappa$, then we would be trying to show that $\rr(\kappa)$ is large for most $\kappa$.
In other words, we might expect to show that if $\kappa$ is not weakly compact, then $\rr(\kappa) = {2}^{\kappa}$.
However it is still an open problem whether the inequality $\rr({\aleph}_{1}) < {2}^{{\aleph}_{1}}$ is consistent (see Question \ref{q:kunen}).
Nevertheless, it is of interest to ask whether for all regular $\kappa > \omega$ the result from \cite{cov_z0} that $\s(\kappa) \leq \bb(\kappa)$ can be dualized to the result that $\dd(\kappa) \leq \rr(\kappa)$.

We present two further $\ZFC$ theorems on cardinal invariants at uncountable regular cardinals in the paper.
Our first result, Theorem \ref{thm:bandakappa}, says that if $\kappa > \omega$ is regular, then $\bb(\kappa) = {\kappa}^{+}$ implies $\inva(\kappa) = {\kappa}^{+}$.
This improves the above mentioned result of Blass, Hyttinen, and Zhang~\cite{BHZ}.
It also shows that $\omega$ is unique among regular cardinals in that it is the only such $\kappa$ where $\bb(\kappa) = {\kappa}^{+} < {\kappa}^{++} = \inva(\kappa)$ is consistent.
Our next result, Theorem \ref{thm:case2}, is a partial dual to our earlier result from \cite{cov_z0}.
It says that for all regular cardinals $\kappa \geq {\beth}_{\omega}$, $\dd(\kappa) \leq \rr(\kappa)$.
Thus for sufficiently large $\kappa$, the invariants $\s(\kappa), \bb(\kappa), \dd(\kappa)$, and $\rr(\kappa)$ are provably comparable and ordered as $\s(\kappa) \leq \bb(\kappa) \leq \dd(\kappa) \leq \rr(\kappa)$.
The proof of our first theorem makes use of the equality $\bb(\kappa) = {\bb}_{\clb}(\kappa)$ of Cummings and Shelah~\cite{sh541} discussed before.
Their theorem that $\dd(\kappa) = {\dd}_{\clb}(\kappa)$ for all regular $\kappa \geq {\beth}_{\omega}$ is not directly used.
However the main idea of the proof of our Theorem \ref{thm:case2} is similar to the main idea in the proof of $\dd(\kappa) = {\dd}_{\clb}(\kappa)$ -- both results use the revised $\GCH$ of Shelah, which is a striking application of $\PCF$ theory exposed in \cite{sh:460}.

Finally one word about our notation, which is standard.
$X \subset Y$ means that $\forall x\[x \in X \implies x \in Y\]$.
So the symbol ``$\subset$'' does not mean ``proper subset''.
If $f$ is a function and $X \subset \dom(f)$, then $f''X$ is the image of $X$ under $f$, that is $f''X = \{f(x): x \in X\}$.
\section{The bounding and almost disjointness numbers: A $\ZFC$ result} \label{sec:bandakappa}
We will quote the following well-known result of Cummings and Shelah~\cite{sh541}.
\begin{Theorem}[see Theorem 6 of \cite{sh541}] \label{thm:bclb}
 For every regular cardinal $k > \omega$, $\bb(\kappa) = {\bb}_{\clb}(\kappa)$.
\end{Theorem}
\begin{Theorem} \label{thm:bandakappa}
 Let $\kappa > \omega$ be a regular cardinal.
 If $\bb(\kappa) = {\kappa}^{+}$, then $\inva(\kappa) = {\kappa}^{+}$.
\end{Theorem}
\begin{proof}
 The hypothesis and Theorem \ref{thm:bclb} imply that there exists a sequence $\seq{f}{\delta}{<}{{\kappa}^{+}}$ of functions in ${\kappa}^{\kappa}$ with the property that for any $g \in {\kappa}^{\kappa}$, there is a $\delta < {\kappa}^{+}$ such that $\{\alpha < \kappa: g(\alpha) < {f}_{\delta}(\alpha)\}$ is stationary in $\kappa$.
 For any $E \subset \kappa$, if $\otp(E) = \kappa$, then let $\langle {\mu}_{E, \xi}: \xi < \kappa \rangle$ be the increasing enumeration of $E$.
 For each $\delta < {\kappa}^{+}$, let ${C}_{\delta} = \{\alpha < \kappa: \alpha \ \text{is closed under} \ {f}_{\delta}\}$.
 Recall that ${C}_{\delta}$ is a club in $\kappa$.
 Also, fix a sequence $\langle {e}_{\delta}: \kappa \leq \delta < {\kappa}^{+}\rangle$ of bijections ${e}_{\delta}: \kappa \rightarrow \delta$.
 We will construct a sequence $\langle \pr{{A}_{\delta}}{{E}_{\delta}}: \delta < {\kappa}^{+}\rangle$ satisfying the following conditions for each $\delta < {\kappa}^{+}$:
 \begin{enumerate}[series=bandakappa]
  \item
  ${A}_{\delta} \in {\[\kappa\]}^{\kappa}$ and ${E}_{\delta} \subset {C}_{\delta}$ is a club in $\kappa$;
  \item
  $\forall \gamma < \delta\[\lc {A}_{\gamma} \cap {A}_{\delta} \rc < \kappa\]$;
  \item
  if $\kappa \leq \delta$, then ${A}_{\delta} = {\bigcup}_{\xi < \kappa}{{B}_{\delta, \xi}}$, where for each $\xi < \kappa$, ${B}_{\delta, \xi}$ is defined to be
  \begin{align*}
   \left\{ {\mu}_{{E}_{\delta}, \xi} \leq \alpha < {\mu}_{{E}_{\delta}, \xi + 1}: \forall \nu < {\mu}_{{E}_{\delta}, \xi}\[\alpha \notin {A}_{{e}_{\delta}(\nu)}\]\right\}.
  \end{align*}
 \end{enumerate}
Suppose for a moment that such a sequence can be constructed.
Let $\A = \{{A}_{\delta}: \delta < {\kappa}^{+}\}$.
By (1) and (2), $\A$ is an a.\@d.\@ family in ${\[\kappa\]}^{\kappa}$ of size ${\kappa}^{+}$.
We claim that it is maximal.
To see this, fix $B \in {\[\kappa\]}^{\kappa}$.
Define a function $g: \kappa \rightarrow \kappa$ by stipulating that for each $\mu \in \kappa$, $g(\mu) = \sup\left( \{\min(B \setminus \left( \mu + 1 \right))\} \cup \{{f}_{\nu}(\mu): \nu \leq \mu\}\right)$.
Find $\delta < {\kappa}^{+}$ such that $S = \{\mu \in \kappa: g(\mu) < {f}_{\delta}(\mu)\}$ is stationary in $\kappa$.
Note that $\kappa \leq \delta$.
Therefore the consequent of (3) applies to $\delta$.
Let $I = \{\xi < \kappa: {B}_{\delta, \xi} \cap B \neq 0\}$.
If $\lc I \rc = \kappa$, then $\lc {A}_{\delta} \cap B \rc = \kappa$, and we are done.
So assume that $\lc I \rc < \kappa$.
Then $\{{\mu}_{{E}_{\delta}, \xi}: \xi \in I\} \subset {E}_{\delta} \subset \kappa$ and $\lc \{{\mu}_{{E}_{\delta}, \xi}: \xi \in I\} \rc \leq \lc I \rc < \kappa$.
Therefore $\sup\left( \{{\mu}_{{E}_{\delta}, \xi}: \xi \in I\} \right) = {\nu}_{0} < \kappa$.
Now $\{\mu \in {E}_{\delta}: \mu > {\nu}_{0}\}$ is a club in $\kappa$ and $T = S \cap \{\mu \in {E}_{\delta}: \mu > {\nu}_{0}\}$ is stationary in $\kappa$.
Consider any $\mu \in T$.
There exists $\xi \in \kappa \setminus I$ with $\mu = {\mu}_{{E}_{\delta}, \xi}$.
Note that ${B}_{\delta, \xi} \cap B = 0$ because $\xi \notin I$.
On the other hand, ${\mu}_{{E}_{\delta}, \xi} = \mu < \min(B \setminus (\mu + 1)) \leq g(\mu) < {f}_{\delta}(\mu) < {\mu}_{{E}_{\delta}, \xi + 1}$ because $\mu \in S$ and because ${\mu}_{{E}_{\delta}, \xi + 1} \in {C}_{\delta}$.
Since $\min(B \setminus (\mu + 1)) \notin {B}_{\delta, \xi}$, it follows from the definition of ${B}_{\delta, \xi}$ that $\exists \nu < \mu\[ \min(B \setminus (\mu + 1)) \in {A}_{{e}_{\delta}(\nu)}\]$.
Thus we have proved that for each $\mu \in T$, $\exists \nu < \mu\exists \beta \in B \[\mu < \beta \wedge \beta \in {A}_{{e}_{\delta}(\nu)}\]$.
Since $T$ is stationary in $\kappa$, there exist ${T}^{\ast} \subset T$ and $\nu$ such that ${T}^{\ast}$ is stationary in $\kappa$ and for each $\mu \in {T}^{\ast}$, $\nu < \mu$ and $\exists \beta \in B \[\mu < \beta \wedge \beta \in {A}_{{e}_{\delta}(\nu)}\]$.
It now easily follows that $\lc {A}_{{e}_{\delta}(\nu)} \cap B \rc = \kappa$.
This proves the maximality of $\A$.
Since $\lc \A \rc = {\kappa}^{+}$, we have $\inva(\kappa) \leq {\kappa}^{+}$, while standard arguments (see Theorem 1.2 of \cite{kunen}) show that ${\kappa}^{+} \leq \inva(\kappa)$.
Hence we have $\inva(\kappa) = {\kappa}^{+}$.

Thus it suffices to construct a sequence satisfying (1)--(3) above.
Let $\seq{A}{\gamma}{\in}{\kappa}$ be any partition of $\kappa$ into $\kappa$ many pairwise disjoint pieces of size $\kappa$.
For each $\gamma < \kappa$, let ${E}_{\gamma} = {C}_{\gamma}$.
It is clear that the sequence $\langle \pr{{A}_{\gamma}}{{E}_{\gamma}}: \gamma < \kappa \rangle$ satisfies (1)--(3).
Now fix ${\kappa}^{+} > \delta \geq \kappa$ and assume that $\langle \pr{{A}_{\gamma}}{{E}_{\gamma}}: \gamma < \delta \rangle$ satisfying (1)--(3) is given.
We construct ${A}_{\delta}$ and ${E}_{\delta}$ as follows.
Let $\theta$ be a sufficiently large regular cardinal.
Let $x = \{\kappa, \seq{f}{\delta}{<}{{\kappa}^{+}}, \seq{C}{\delta}{<}{{\kappa}^{+}}, \langle {e}_{\delta}: \kappa \leq \delta < {\kappa}^{+}\rangle, \delta, \langle \pr{{A}_{\gamma}}{{E}_{\gamma}}: \gamma < \delta \rangle\}$.
Let $\seq{N}{\xi}{<}{\kappa}$ be such that
\begin{enumerate}[resume=bandakappa]
 \item
 $\forall \xi < \kappa\[{N}_{\xi} \prec H(\theta) \wedge x \in {N}_{\xi}\]$;
 \item
 $\forall \xi < \kappa\[\lc {N}_{\xi} \rc < \kappa \wedge {\mu}_{\xi} = {N}_{\xi} \cap \kappa \in \kappa\]$;
 \item
 $\forall \xi < \xi + 1 < \kappa\[\seq{N}{\zeta}{\leq}{\xi} \in {N}_{\xi + 1}\]$;
 \item
 $\forall \xi < \kappa\[\xi \ \text{is a limit ordinal} \implies {N}_{\xi} = {\bigcup}_{\zeta < \xi}{{N}_{\zeta}}\]$.
\end{enumerate}
Observe that these conditions imply that $\forall \zeta < \xi < \kappa\[{N}_{\zeta} \in {N}_{\xi} \wedge {N}_{\zeta} \subset {N}_{\xi}\]$.
Observe also that ${E}_{\delta} = \{{\mu}_{\xi}: \xi < \kappa\}$ is a club in $\kappa$ and that ${\mu}_{{E}_{\delta}, \xi} = {\mu}_{\xi}$, for all $\xi < \kappa$.
Next for each $\xi < \kappa$, ${C}_{\delta} \in {N}_{\xi}$.
It follows that ${\mu}_{\xi} \in {C}_{\delta}$ because ${C}_{\delta}$ is a club in $\kappa$.
So ${E}_{\delta} \subset {C}_{\delta}$.
Now define ${A}_{\delta} = {\bigcup}_{\xi < \kappa}{{B}_{\delta, \xi}}$, where for each $\xi < \kappa$, ${B}_{\delta, \xi}$ is
\begin{align*}
 \left\{{\mu}_{\xi} \leq \alpha < {\mu}_{\xi + 1}: \forall \nu < {\mu}_{\xi}\[\alpha \notin {A}_{{e}_{\delta}(\nu)}\]\right\}.
\end{align*}
It is clear that (3) is satisfied by definition and that ${A}_{\delta} \subset \kappa$.
So to complete the proof, it suffices to check that $\lc {A}_{\delta} \rc = \kappa$ and that $\forall \gamma < \delta\[\lc {A}_{\gamma} \cap {A}_{\delta} \rc < \kappa\]$.
To see the second statement, fix any $\gamma < \delta$.
Since ${e}_{\delta}: \kappa \rightarrow \delta$ is a bijection, we can find $\nu \in \kappa$ with ${e}_{\delta}(\nu) = \gamma$.
Find $\zeta < \kappa$ with $\nu < {\mu}_{\zeta}$.
Consider any $\xi < \kappa$ so that $\zeta \leq \xi$.
Then $\nu < {\mu}_{\zeta} \leq {\mu}_{\xi}$.
It follows that ${A}_{\gamma} \cap {B}_{\delta, \xi} = {A}_{{e}_{\delta}(\nu)} \cap {B}_{\delta, \xi} = 0$.
Therefore, ${A}_{\gamma} \cap {A}_{\delta} = {\bigcup}_{\xi < \kappa}{\left( {A}_{\gamma} \cap {B}_{\delta, \xi} \right)} = {\bigcup}_{\xi < \zeta}{\left( {A}_{\gamma} \cap {B}_{\delta, \xi} \right)} \subset {\bigcup}_{\xi < \zeta}{{B}_{\delta, \xi}}$.
For each $\xi < \zeta$, $\lc {B}_{\delta, \xi} \rc < \kappa$.
So ${\bigcup}_{\xi < \zeta}{{B}_{\delta, \xi}}$ is the union of $\leq \lc \zeta \rc \leq \zeta < \kappa$ many sets each of size $< \kappa$.
Since $\kappa$ is regular, we conclude that $\lc {\bigcup}_{\xi < \zeta}{{B}_{\delta, \xi}} \rc < \kappa$.
So $\lc {A}_{\gamma} \cap {A}_{\delta} \rc < \kappa$, as needed.

Finally we check that for each $\xi < \kappa$, ${B}_{\delta, \xi} \neq 0$.
This will imply that $\lc {A}_{\delta} \rc = \kappa$.
Fix any $\xi < \kappa$.
Note that for each $\nu < {\mu}_{\xi}$, $\lc {A}_{{e}_{\delta}({\mu}_{\xi})} \cap {A}_{{e}_{\delta}(\nu)} \rc < \kappa$.
Therefore ${R}_{\xi} = {\bigcup}_{\nu < {\mu}_{\xi}}{\left( {A}_{{e}_{\delta}({\mu}_{\xi})} \cap {A}_{{e}_{\delta}(\nu)} \right)}$ is the union of at most $\lc {\mu}_{\xi} \rc \leq {\mu}_{\xi} < \kappa$ many sets each having size $< \kappa$.
Since $\kappa$ is regular, it follows that $\lc {R}_{\xi} \rc < \kappa$.
Hence there is an $\alpha \in {A}_{{e}_{\delta}({\mu}_{\xi})} \setminus {R}_{\xi}$ with ${\mu}_{\xi} \leq \alpha$ because $\lc {A}_{{e}_{\delta}({\mu}_{\xi})} \rc = \kappa$.
Since ${N}_{\xi + 1} \prec H(\theta)$ and since all the relevant parameters belong to ${N}_{\xi + 1}$, we conclude that there exists $\alpha \in {N}_{\xi + 1}$ such that $\alpha \in \kappa$, ${\mu}_{\xi} \leq \alpha$, and $\forall \nu \in {\mu}_{\xi}\[\alpha \notin {A}_{{e}_{\delta}(\nu)}\]$.
Now we have that ${\mu}_{\xi} \leq \alpha < {\mu}_{\xi + 1}$ and so $\alpha \in {B}_{\delta, \xi}$.
This shows that ${B}_{\delta, \xi} \neq 0$ and concludes the proof.
\end{proof}
\section{The reaping and dominating numbers: an application of PCF theory} \label{sec:randd}
We begin with a well-known fact, whose proof we include for completeness.
\begin{Def} \label{def:es}
  Let $\kappa > \omega$ be a regular cardinal.
  If $A \in {\[\kappa\]}^{\kappa}$, then we let ${e}_{A}: \kappa \rightarrow A$ be the order isomorphism from $\pr{\kappa}{\in}$ to $\pr{A}{\in}$.
  We also define a function ${s}_{A}: \kappa \rightarrow A$ by setting ${s}_{A}(\alpha) = \min(A \setminus (\alpha + 1))$, for each $\alpha \in \kappa$.
  We also write $\lim(\kappa) = \{\alpha < \kappa: \alpha \ \text{is a limit ordinal} \}$ and $\succc(\kappa) = \{\alpha < \kappa: \alpha \ \text{is a successor ordinal} \}$.
\end{Def}
\begin{Lemma}[Folklore] \label{lem:risbig}
  If $\kappa > \omega$ is a regular cardinal, then $\rr(\kappa) \geq {\kappa}^{+}$.
\end{Lemma}
\begin{proof}
Let $F \subset {\[\kappa\]}^{\kappa}$ be a family with $\lc F \rc \leq \kappa$.
We must find a $B \in \Pset(\kappa)$ which reaps $F$.
If $F$ is empty, then $B = \kappa$ will work.
So assume $F$ is non-empty.
Let $\{{A}_{\alpha}: \alpha < \kappa \}$ enumerate $F$, possibly with repetitions.
For each $\alpha < \kappa$, let ${C}_{\alpha} = \{\delta < \kappa: \delta \ \text{is closed under} \ {s}_{{A}_{\alpha}} \}$.
Then $C = \{\delta < \kappa: \forall \alpha < \delta\[\delta \in {C}_{\alpha} \]\}$ is a club in $\kappa$.
For each $\xi \in \kappa$, let ${B}_{\xi} = \{ \zeta < {e}_{C}(\xi + 1): {e}_{C}(\xi) \leq \zeta \}$.
Note that for all $\alpha < {e}_{C}(\xi + 1)$, ${A}_{\alpha} \cap {B}_{\xi} \neq 0$.
Also for any distinct $\xi, \xi' \in \kappa$, ${B}_{\xi} \cap {B}_{\xi'} = 0$.
Put $B = \bigcup\left\{ {B}_{\xi}: \xi \in \lim(\kappa) \right\}$.
Then $B \in \Pset(\kappa)$ and since for each $\alpha < \kappa$ and each $\xi \in \lim(\kappa) \setminus \alpha$, ${A}_{\alpha} \cap {B}_{\xi} \neq 0$, $\lc {A}_{\alpha} \cap B \rc = \kappa$, for all $\alpha < \kappa$.
Furthermore, $\bigcup\left\{ {B}_{\xi'}: \xi' \in \succc(\kappa) \right\} \subset \kappa \setminus B$, and since for each $\alpha < \kappa$ and for each $\xi' \in \succc(\kappa) \setminus \alpha$, ${A}_{\alpha} \cap {B}_{\xi'} \neq 0$, $\lc {A}_{\alpha} \cap \left( \kappa \setminus B \right) \rc = \kappa$, for all $\alpha < \kappa$.
Thus $B$ reaps $F$.
\end{proof}
The above proof really shows that $\rr(\kappa) \geq \bb(\kappa)$.
However we will not need this in what follows.
The proof of the main theorem is broken into two cases.
For the remainder of this section, let $\kappa > \omega$ be a fixed regular cardinal.
The crucial definition is the following.
\begin{Def} \label{def:set}
  Let ${E}_{2} \subset {E}_{1}$ both be clubs in $\kappa$.
  For each $\xi \in \kappa$, define $\set({E}_{1}, \xi) = \left\{ \zeta < {s}_{{E}_{1}}(\xi): \xi \leq \zeta \right\}$.
  Define $\set({E}_{2}, {E}_{1}) = \bigcup\left\{ \set({E}_{1}, \xi): \xi \in {E}_{2} \right\}$.
\end{Def}
\begin{Lemma} \label{lem:case1}
  Suppose that $F \subset {\[\kappa\]}^{\kappa}$ is an unreaped family with $\lc F \rc = \rr(\kappa)$.
  Assume there is a club ${E}_{1} \subset \kappa$ such that for each club $E \subset {E}_{1}$, there exists $A \in F$ with $A \: {\subset}^{\ast} \: \set(E, {E}_{1})$.
  Then $\dd(\kappa) \leq \rr(\kappa)$.
\end{Lemma}
\begin{proof}
  For each $A \in F$ define a function ${g}_{A}: \kappa \rightarrow \kappa$ as follows.
  Given $\beta \in \kappa$, ${g}_{A}(\beta) = {s}_{A}({s}_{{E}_{1}}(\beta))$.
  Then $\lc \{{g}_{A}: A \in F \} \rc \leq \lc F \rc = \rr(\kappa)$, and we will check that this is a dominating family of functions.
  To this end, fix any $f \in {\kappa}^{\kappa}$.
  Put
  \begin{align*}
    {E}_{f} = \left\{ \xi \in {E}_{1}: \xi \ \text{is closed under} \ f \right\}.
  \end{align*}
  Then ${E}_{f} \subset {E}_{1}$ and it is a club in $\kappa$.
  By hypothesis there exist $A \in F$ and $\delta \in \kappa$ with $A \setminus \delta \subset \set({E}_{f}, {E}_{1})$.
  We claim that for any $\zeta \in \kappa$, if $\zeta \geq \delta$, then $f(\zeta) < {g}_{A}(\zeta)$.
  Indeed suppose $\delta \leq \zeta < \kappa$ is given.
  Let $\gamma = {s}_{{E}_{1}}(\zeta) > \zeta$ and let ${g}_{A}(\zeta) = \beta = {s}_{A}({s}_{{E}_{1}}(\zeta))$.
  Then $\beta \in A$ and $\delta \leq \zeta < {s}_{{E}_{1}}(\zeta) < \beta$.
  Thus $\beta \in \set({E}_{f}, {E}_{1})$.
  Let $\zeta' \in {E}_{f}$ be such that $\zeta' \leq \beta < {s}_{{E}_{1}}(\zeta')$.
  It could not be the case that $\zeta' < \gamma$, for if that were the case, then the inequality $\beta < {s}_{{E}_{1}}(\zeta') \leq \gamma = {s}_{{E}_{1}}(\zeta) < \beta$ would be true, which is impossible. 
  Therefore $\gamma \leq \zeta'$ and since $\zeta < \gamma \leq \zeta'$ and $\zeta'$ is closed under $f$, we have $f(\zeta) < \zeta' \leq \beta = {g}_{A}(\zeta)$, as claimed.
  Hence $f \: {\leq}^{\ast} \: {g}_{A}$.
  As $f \in {\kappa}^{\kappa}$ was arbitrary, this proves that $\{{g}_{A}: A \in F \}$ is dominating, and so $\dd(\kappa) \leq \lc \{{g}_{A}: A \in F \} \rc \leq \rr(\kappa)$. 
\end{proof}
The proof in the case when the hypothesis of Lemma \ref{lem:case1} fails will make use of Shelah's Revised $\GCH$, which is a theorem of $\ZFC$.
Let us recall the definition of various notions that are relevant to the revised $\GCH$.
\begin{Def} \label{def:weakpower}
  Let $\kappa$ and $\lambda$ be cardinals.
  Define ${\lambda}^{\[\kappa\]}$ to be
  \begin{align*}
    \min\left\{ \lc \PPP \rc: \PPP \subset {\[\lambda\]}^{\leq \kappa} \ \text{and} \ \forall u \in {\[\lambda\]}^{\kappa}\exists{\PPP}_{0} \subset \PPP\[\lc {\PPP}_{0} \rc < \kappa \ \text{and} \ u = \bigcup {\PPP}_{0} \] \right\}.
  \end{align*}
  The operation ${\lambda}^{\[\kappa\]}$ is sometimes referred to as the \emph{weak power}.
\end{Def}
The following remarkable $\ZFC$ result was obtained by Shelah in \cite{sh:460} as one of the many fruits of his $\PCF$ theory.
A nice exposition of its proof may be also be found in Abraham and Magidor~\cite{handbookPCF}.
Another relevant reference is Shelah~\cite{sh:829}.
\begin{Theorem}[The Revised $\GCH$]\label{thm:revisedgch}
  If $\theta$ is a strong limit uncountable cardinal, then for every $\lambda \geq \theta$, there exists $\sigma < \theta$ such that for every $\sigma \leq \kappa < \theta$, ${\lambda}^{\[\kappa\]} = \lambda$.
\end{Theorem}
\begin{Cor} \label{cor:revisedgch}
  Let $\mu \geq {\beth}_{\omega}$ be any cardinal.
  There exists an uncountable regular cardinal $\theta < {\beth}_{\omega}$ and a family $\PPP \subset {\[\mu\]}^{\leq \theta}$ such that $\lc \PPP \rc \leq \mu$ and for each $u \in {\[\mu\]}^{\theta}$, there exists $v \in \PPP$ with the property that $v \subset u$ and $\lc v \rc \geq {\aleph}_{0}$.
\end{Cor}
\begin{proof}
${\beth}_{\omega}$ is a strong limit uncountable cardinal.
Therefore Theorem \ref{thm:revisedgch} applies and implies that there exists $\sigma < {\beth}_{\omega}$ such that for every $\sigma \leq \theta < {\beth}_{\omega}$, ${\mu}^{\[\theta\]} = \mu$.
It is possible to choose an uncountable regular cardinal $\theta$ satisfying $\sigma \leq \theta < {\beth}_{\omega}$.
Since ${\mu}^{\[\theta\]} = \mu$, there exists $\PPP \subset {\[\mu\]}^{\leq \theta}$ such that $\lc \PPP \rc = \mu$ and for each $u \in {\[\mu\]}^{\theta}$, there exists ${\PPP}_{0} \subset \PPP$ with the property that $\lc {\PPP}_{0} \rc < \theta$ and $u = \bigcup {\PPP}_{0}$.
Now suppose that $u \in {\[\mu\]}^{\theta}$ is given.
Let ${\PPP}_{0} \subset \PPP$ be such that $\lc {\PPP}_{0} \rc < \theta$ and $u = \bigcup {\PPP}_{0}$.
Since $\theta$ is a regular cardinal and $\lc u \rc = \theta$, it follows that $\lc v \rc = \theta \geq {\aleph}_{0}$, for some $v \in {\PPP}_{0}$.
This is as required because $v \in \PPP$ and $v \subset u$.
\end{proof}
The proof of the following theorem is similar to the proof of Cummings and Shelah's theorem from~\cite{sh541} that if $\kappa \geq {\beth}_{\omega}$, then $\dd(\kappa) = {\dd}_{\clb}(\kappa)$.
\begin{Theorem} \label{thm:case2}
  If $\kappa \geq {\beth}_{\omega}$, then $\dd(\kappa) \leq \rr(\kappa)$.
\end{Theorem}
\begin{proof}
  Write $\mu = \rr(\kappa)$.
  Let $F \subset {\[\kappa\]}^{\kappa}$ be such that $F$ is unreaped and $\lc F \rc = \mu$.
  Then ${\beth}_{\omega} \leq \kappa < {\kappa}^{+} \leq \rr(\kappa) = \mu$.
  So applying Corollary \ref{cor:revisedgch}, fix an uncountable regular cardinal $\theta < {\beth}_{\omega}$ satisfying the conclusion of Corollary \ref{cor:revisedgch}.
  Note that $\lc \theta \times \mu \rc = \mu$ because $\theta < {\beth}_{\omega} < \mu$.
  So $\lc \theta \times F \rc = \mu$.
  Therefore applying Corollary \ref{cor:revisedgch}, find a family $\PPP \subset {\[\theta \times F\]}^{\leq \theta}$ such that $\lc \PPP \rc \leq \mu$ and $\PPP$ has the property that for each $u \in {\[\theta \times F\]}^{\theta}$, there exists $v \in \PPP$ satisfying $v \subset u$ and $\lc v \rc \geq {\aleph}_{0}$.
  Put $X = F \cup \mu \cup \PPP \cup \{\theta, \mu, \kappa, {\kappa}^{\kappa}, \Pset(\kappa)\}$.
  Then $\lc X \rc = \mu$, and so if $\chi$ is a sufficiently large regular cardinal, then there exists $M \prec H(\chi)$ with $\lc M \rc = \mu$ and $X \subset M$.
  We will aim to prove that $M \cap {\kappa}^{\kappa}$ is a dominating family.
  
  In view of Lemma \ref{lem:case1} it may be assumed that for any club ${E}_{1} \subset \kappa$, there exists a club ${E}_{2} \subset {E}_{1}$ such that for all $B \in F$, $B \: {\not\subset}^{\ast} \: \set({E}_{2}, {E}_{1})$.
  Since $F$ is an unreaped family and since $\set({E}_{2}, {E}_{1}) \in \Pset(\kappa)$ whenever ${E}_{2} \subset {E}_{1}$ are both clubs in $\kappa$, it follows that for each club ${E}_{1} \subset \kappa$, there exist a club ${E}_{2} \subset {E}_{1}$ and a $B \in F$ such that $B \: {\subset}^{\ast} \: \kappa \setminus \set({E}_{2}, {E}_{1})$.
  Let $f \in {\kappa}^{\kappa}$ be a fixed function.
  Construct a sequence $\langle \pr{{E}_{i}}{{E}^{1}_{i}, {B}_{i}}: i < \theta \rangle$ by induction on $i < \theta$ so that the following conditions are satisfied at each $i < \theta$:
  \begin{enumerate}
    \item
    ${E}_{i}$ and ${E}^{1}_{i}$ are both clubs in $\kappa$, ${E}^{1}_{i} \subset {E}_{i}$, and $\forall j < i\[{E}_{i} \subset {E}^{1}_{j}\]$;
    \item
    ${B}_{i} \in F$ and ${B}_{i} \: {\subset}^{\ast} \: \kappa \setminus \set({E}^{1}_{i}, {E}_{i})$;
    \item
    if $i = 0$, then ${E}_{i} = \{\alpha < \kappa: \alpha \ \text{is closed under} \ f \}$.
  \end{enumerate}
  We first show how to construct such a sequence.
  When $i = 0$, put ${E}_{i} = \{\alpha < \kappa: \alpha \ \text{is closed under} \ f \}$.
  Then ${E}_{i}$ is a club in $\kappa$, and so there exist a club ${E}^{1}_{i} \subset {E}_{i}$ and a ${B}_{i} \in F$ with ${B}_{i} \: {\subset}^{\ast} \: \kappa \setminus \set({E}^{1}_{i}, {E}_{i})$.
  Next suppose that $\theta > i > 0$ and that $\langle \pr{{E}_{j}}{{E}^{1}_{j}, {B}_{j}}: j < i \rangle$ satisfying (1)--(3) is given.
  Then $\{{E}^{1}_{j}: j < i\}$ is a collection of $\leq \lc i \rc \leq i < \theta < {\beth}_{\omega} \leq \kappa$ many clubs in $\kappa$.
  Therefore ${E}_{i} = {\bigcap}_{j < i}{{E}^{1}_{j}}$ is a club in $\kappa$.
  We have $\forall j < i\[{E}_{i} \subset {E}^{1}_{j}\]$ and moreover there exist a club ${E}^{1}_{i} \subset {E}_{i}$ and a ${B}_{i} \in F$ such that ${B}_{i} \: {\subset}^{\ast} \: \kappa \setminus \set({E}^{1}_{i}, {E}_{i})$.
  It is clear that ${E}_{i}, {E}^{1}_{i}$, and ${B}_{i}$ are as required.
  This completes the construction of the sequence $\langle \pr{{E}_{i}}{{E}^{1}_{i}, {B}_{i}}: i < \theta \rangle$.
  
  Now define a function $u: \theta \rightarrow F$ by setting $u(i) = {B}_{i}$ for all $i \in \theta$.
  Then $u \subset \theta \times F$ and $\lc u \rc = \lc \dom(u) \rc = \theta$.
  Hence by the choice of $\PPP$ and $M$, there exists $v \in \PPP \subset X \subset M$ such that $v \subset u$ and $\lc v \rc \geq {\aleph}_{0}$.
  $v$ is a function and $c = \dom(v) \subset \dom(u) = \theta$.
  Moreover, ${\aleph}_{0} \leq \lc v \rc = \lc c \rc$ and $c \in M$.
  Hence we can find $d \in M$ so that $d \subset c$ and $\otp(d) = \omega$.
  Let $w = v \restrict d \in M$.
  Since $\kappa > \omega$ is regular, there exists a function $g \in {\kappa}^{\kappa}$ with the property that for each $\alpha \in \kappa$, $\forall i \in d \exists \beta \in w(i) = {B}_{i}\[\alpha < \beta < g(\alpha)\]$.
  We may further assume that $g \in M$ because all of the relevant parameters belong to $M$.
  Let $\seq{i}{n}{\in}{\omega}$ be the strictly increasing enumeration of $d$.
  Recall that for each $n \in \omega$, ${E}^{1}_{{i}_{n}} \subset {E}_{{i}_{n}} \subset \kappa$ are both clubs in $\kappa$ and that ${B}_{{i}_{n}} \: {\subset}^{\ast} \: \kappa \setminus \set({E}^{1}_{{i}_{n}}, {E}_{{i}_{n}})$.
  In particular, for each $n \in \omega$, there exists ${\delta}_{n} \in \kappa$ so that ${B}_{{i}_{n}} \setminus {\delta}_{n} \subset \kappa \setminus \set({E}^{1}_{{i}_{n}}, {E}_{{i}_{n}})$, and also $\min({E}_{{i}_{n}}) \in \kappa$.
  Hence $\{{\delta}_{n}: n \in \omega\} \cup \{\min({E}_{{i}_{n}}): n \in \omega\}$ is a countable subset of $\kappa$, whence $\{{\delta}_{n}: n \in \omega\} \cup \{\min({E}_{{i}_{n}}): n \in \omega\} \subset \delta$, for some $\delta \in \kappa$.
  We will argue that for each $\alpha \in \kappa$, if $\alpha \geq \delta$, then $f(\alpha) < g(\alpha)$.
  To this end, let $\alpha \in \kappa$ be fixed, and assume that $\delta \leq \alpha$.
  For each $n \in \omega$, since ${E}_{{i}_{n}} \subset \kappa$ is a club in $\kappa$ and since $\min({E}_{{i}_{n}}) < \delta \leq \alpha < \alpha + 1 < \kappa$, it follows that ${\xi}_{n} = \sup({E}_{{i}_{n}} \cap \left( \alpha + 1 \right) ) \in {E}_{{i}_{n}}$.
  Also $\forall n \in \omega\[{\xi}_{n + 1} \leq {\xi}_{n}\]$ because $\forall n \in \omega\[{E}_{{i}_{n + 1}} \subset {E}_{{i}_{n}}\]$.
  It follows that there exist $\xi$ and $N \in \omega$ such that $\forall n \geq N\[{\xi}_{n} = \xi\]$.
  Note that $\xi \in {E}_{{i}_{N + 1}} \subset {E}^{1}_{{i}_{N}}$.
  Consider ${s}_{{E}_{{i}_{N}}}(\xi)$.
  ${s}_{{E}_{{i}_{N}}}(\xi) \in {E}_{{i}_{N}}$ and ${s}_{{E}_{{i}_{N}}}(\xi) > \xi = {\xi}_{N} = \sup({E}_{{i}_{N}} \cap \left( \alpha + 1 \right) )$.
  Therefore ${s}_{{E}_{{i}_{N}}}(\xi) \geq \alpha + 1 > \alpha$.
  Since ${s}_{{E}_{{i}_{N}}}(\xi) \in {E}_{{i}_{N}} \subset {E}_{0}$, ${s}_{{E}_{{i}_{N}}}(\xi)$ is closed under $f$.
  Therefore $f(\alpha) < {s}_{{E}_{{i}_{N}}}(\xi)$.
  Next by the choice of $g$, there exists $\beta \in {B}_{{i}_{N}}$ with $\alpha < \beta < g(\alpha)$.
  Note that ${\delta}_{N} < \delta \leq \alpha < \beta$.
  Hence $\beta \in {B}_{{i}_{N}} \setminus {\delta}_{N} \subset \kappa \setminus \set({E}^{1}_{{i}_{N}}, {E}_{{i}_{N}})$, in other words, $\beta \notin \set({E}^{1}_{{i}_{N}}, {E}_{{i}_{N}})$.
  Note that $\xi = \sup({E}_{{i}_{N}} \cap \left( \alpha + 1 \right) ) \leq \alpha < \beta$.
  Since $\xi \in {E}^{1}_{{i}_{N}}$, $\beta \geq {s}_{{E}_{{i}_{N}}}(\xi)$.
  Putting all this information together, we have $f(\alpha) < {s}_{{E}_{{i}_{N}}}(\xi) \leq \beta < g(\alpha)$, as required.
  
  Thus we have proved that $f \: {\leq}^{\ast} \: g$.
  Since $f \in {\kappa}^{\kappa}$ was arbitrary and since $g \in M \cap {\kappa}^{\kappa}$, we have proved that $M \cap {\kappa}^{\kappa}$ is a dominating family.
  Therefore $\dd(\kappa) \leq \lc M \cap {\kappa}^{\kappa} \rc \leq \lc M \rc = \mu = \rr(\kappa)$.
\end{proof}
\section{Questions} \label{sec:questions}
It is unknown how large $\bb(\kappa)$ needs to be for the configuration $\bb(\kappa) < \inva(\kappa)$ to be consistent.
So we ask
\begin{Question} \label{q:b<a}
  Does $\bb(\kappa) = {\kappa}^{++}$ imply that $\inva(\kappa) = {\kappa}^{++}$, for every regular cardinal $\kappa > \omega$?
\end{Question}
It is not possible to step-up the proof of Theorem \ref{thm:bandakappa} in any straightforward way.
If Question \ref{q:b<a} has a positive answer, then the proof is likely to involve quite a different argument.

Theorem \ref{thm:case2} of course gives no information about the relationship between $\dd(\kappa)$ and $\rr(\kappa)$ when $\kappa < {\beth}_{\omega}$.
\begin{Question} \label{q:rsmall}
  If $\omega < \kappa < {\beth}_{\omega}$ is a regular cardinal, then does $\dd(\kappa) \leq \rr(\kappa)$ hold?
  In particular, is $\dd({\aleph}_{n}) \leq \rr({\aleph}_{n})$, for all $1 \leq n < \omega$?
\end{Question}
In trying to tackle this problem, it may seem reasonable to first try to produce a model where $\rr({\aleph}_{n}) < {2}^{{\aleph}_{n}}$, for if $\rr({\aleph}_{n})$ is provably equal to ${2}^{{\aleph}_{n}}$, then of course $\dd({\aleph}_{n}) \leq \rr({\aleph}_{n})$.  
This is closely related to a well-known question of Kunen about the minimal size of a base for a uniform ultrafilter on ${\aleph}_{1}$.
\begin{Question} \label{q:kunen}
  Is $\rr({\aleph}_{1}) < {2}^{{\aleph}_{1}}$ consistent?
  Is $\uu({\aleph}_{1}) < {2}^{{\aleph}_{1}}$ consistent?
\end{Question}
\def\polhk#1{\setbox0=\hbox{#1}{\ooalign{\hidewidth
  \lower1.5ex\hbox{`}\hidewidth\crcr\unhbox0}}}
\providecommand{\bysame}{\leavevmode\hbox to3em{\hrulefill}\thinspace}
\providecommand{\MR}{\relax\ifhmode\unskip\space\fi MR }
\providecommand{\MRhref}[2]{%
  \href{http://www.ams.org/mathscinet-getitem?mr=#1}{#2}
}
\providecommand{\href}[2]{#2}

%\bibliographystyle{amsplain}
%\bibliography{Bibliography} 
\end{document}